\documentclass[11pt]{amsart}
\usepackage{amssymb,latexsym,graphicx, amscd}
\newtheorem{theorem}{Theorem}[section]
\newtheorem{prop}[theorem]{Proposition}
\newtheorem{lemma}[theorem]{Lemma}
\newtheorem{remark}[theorem]{Remark}
\newtheorem{question}[theorem]{Question}

\begin{document}
\title[Exact forms on almost complex manifolds]{A note on exact forms on almost complex manifolds}
\author{Tedi Draghici}
\address{Department of Mathematics\\Florida International University\\Miami, FL 33199}
\email{draghici@fiu.edu}

\author{Weiyi Zhang}
\address{Department  of Mathematics\\  University of Michigan\\ Ann Arbor, MI 48109}
\email{wyzhang@umich.edu}

 \begin{abstract}
 Reformulations of Donaldson's ``tamed to compatible" question are obtained in terms
 of spaces of exact forms on a compact almost complex manifold $(M^{2n},J)$.
 In dimension 4, we show that $J$ admits a compatible symplectic form if and only if $J$ admits
 tamed symplectic forms with arbitrarily given $J$-anti-invariant parts.
 Some observations about the cohomology of $J$-modified de Rham complexes are also made.
 \end{abstract}
\maketitle

\section{Introduction}

Among other interesting problems raised in \cite{D}, Donaldson asked the following question for a compact
almost complex 4-manifold $(M^4,J)$:
\begin{question} \label{Donaldson} If $J$ is tamed by a symplectic form, is there a symplectic form
compatible with $J$?
\end{question}
\noindent An almost complex structure $J$ on a manifold $M^{2n}$ is {\em tamed} by a
symplectic form $\omega$ (and such an $\omega$ is called $J$-{\em tamed}), if $\omega$ is $J$-positive, i.e.
$$ \omega(X,JX) > 0, \quad \forall X \in T M, \; X \neq 0 . $$
A symplectic form  $\omega$ is {\em compatible} with $J$ (or $J$-{\em compatible}), if $\omega$ is $J$-positive and $J$-invariant, i.e.
$$ \omega(X,JX) > 0 \mbox{ and } \omega(JY,JZ) = \omega(Y,Z),  \quad \forall X, Y, Z \in T M, \; X \neq 0 . $$

From deep works of Taubes and Gromov, Question \ref{Donaldson} was known to have an affirmative answer
on $\mathbb CP^2$. Recently, Taubes \cite{Tau09} showed that the same is true on all compact 4-manifolds with $b^+=1$,
for generic almost complex
structures inducing the given orientation. Some extensions of the results of Taubes are obtained in \cite{LZ2010}.
In particular, it is shown that Question \ref{Donaldson} is true for all almost complex structures on
$\mathbb CP^2\# \overline{\mathbb CP^2}$ and $S^2 \times S^2$.
The question could be asked for higher dimensions as well. It is known however that in dimensions
higher than 4, certain almost complex structures have local obstructions,
coming from the structure of their Nijenhuis tensor, to admitting compatible symplectic forms (see e.g. \cite{Lej}).
Such local obstructions do not exist in dimension 4, or for integrable almost complex structures
in any dimension. Question \ref{Donaldson} was raised for compact complex manifolds of arbitrary dimensions in \cite{LZ2007} and
\cite{ST}. It was answered affirmatively for compact complex surfaces in \cite{LZ2007} 
(see also \cite{DLZ2} for a different proof). Some positive results are known for higher dimensions
(see e.g. \cite{EFV}, \cite{Pet}), but the problem is still open in this case. Note that Theorem 1.4 of \cite{Pet} implies that any
non-K\"ahler Moishezon manifold is also non-tamed.
Even for complex surfaces, there is interest in finding another proof for Question \ref{Donaldson}.
The existing proofs use the celebrated result that a compact complex surface admits a K\"ahler metric
if and only if the first Betti number is even. As Donaldson pointed out, a direct solution for the question
would yield a different proof of this fundamental result.

\vspace{0.1cm}

In this note we describe some reformulations of Question \ref{Donaldson} in terms of
certain spaces of exact forms on an almost complex manifold $(M^{2n},J)$. As application,
we prove a result which can be thought as a further partial answer to Donaldson's question in dimension 4.
\begin{theorem} \label{main}
Let $(M^4,J)$ be a compact almost complex manifold. The following are equivalent:

(i) $J$ admits a compatible symplectic form;

(ii) For any $J$-anti-invariant form $\alpha$, there exists a $J$-tamed symplectic form
whose $J$-anti-invariant part is $\alpha$.
\end{theorem}
\noindent A higher dimensional version of this result is given in Theorem \ref{mainhd}.

\vspace{0.2cm} {\bf Acknowledgments:} We are grateful to Tian-Jun Li for useful
suggestions about this note.
The second author would also like to thank the organizers of the G\"okova Geometry Topology
Conference for providing excellent atmosphere of research.

\section{Exact forms on almost complex 4-manifolds}
Let $(M^{2n}, J)$ be an almost complex manifold. The almost complex
structure $J$ acts as an involution, by $\alpha(\cdot, \cdot)
\rightarrow \alpha(J\cdot, J\cdot)$, on the space of real $C^{\infty}$-forms $\Omega^2$.
Thus we have the splitting
into $J$-invariant,
respectively, $J$-anti-invariant 2-forms
\begin{equation} \label{formtype}
\Omega^2=\Omega_J^+\oplus \Omega_J^-.
\end{equation}
We denote by $\mathcal Z^k$ the space of closed $k$-forms on $M$ and by $\mathcal
Z_J^{\pm} = \mathcal Z^2 \cap \Omega_J^{\pm}$ the corresponding
subspaces of $\mathcal Z^2$.
The sub-groups $H_J^+$, $H_J^-$ of $H^2(M, \mathbb{R})$,
\begin{equation} \nonumber
H_J^{\pm}=\{ \mathfrak{a} \in H^2(M;\mathbb R) | \exists \; \alpha\in \mathcal
Z_J^{\pm} \mbox{ such that } [\alpha] = \mathfrak{a} \} \,
\end{equation}
and their dimensions $h^{\pm}_J = {\rm dim}(H_J^{\pm})$
are interesting invariants of the almost complex manifold $(M^{2n},J)$. They play a secondary role in our note,
but the reader can consult the references for more on these subgroups \cite{LZ2007, FT, DLZ, DLZ2, AT}.
Here we just recall the particularity of dimension 4 for these subgroups; namely, for a compact 4-dimensional manifold $M^4$
and any almost complex structure $J$ on $M^4$, the subgroups  $H_J^+$, $H_J^-$ induce a direct sum
decomposition of $H^2(M^4, \mathbb{R})$. In particular, $h^+_J + h^-_J = b_2$, where $b_2$ is the second Betti number
of $M^4$. These facts are no longer true in higher dimensions for general almost complex structures.
If $b^{+}$ (resp. $b^-$) are the ``self-dual'' (resp. ``anti-self-dual'') Betti numbers of a compact
manifold $(M^4, J)$, it is also known that $h^-_J \leq b^+$, $h^+_J \geq b^-$, with the inequalities being strict
if $J$ is tamed by a symplectic form \cite{DLZ}.

\vspace{0.2cm}

We begin with the following proposition.
\begin{prop}\label{lemma4d} Suppose $(M^4, J)$ is a compact almost complex 4-manifold. Then
\begin{equation} \label{4dincl}
d\Omega_J^-\subsetneq d\Omega^2 = d\Omega_J^+ \;.
\end{equation}
Moreover, the quotient space $(d\Omega^2)/(d\Omega_J^-)$ is always infinite dimensional.
\end{prop}

\begin{proof} Let $g$ be a Riemannian metric compatible with $J$ and let $\omega$ denote the fundamental form
of $(g,J)$. It is well known that
\begin{equation}\label{type-sdasd}
\Omega_J^+= C^{\infty}(M)\omega \oplus \Omega_g^- \; , \; \;
\Omega_g^+ = C^{\infty}(M)\omega \oplus \Omega_J^-,
\end{equation}
where $\Omega_g^{\pm}$ denote the spaces of $g$-self-dual (resp. anti-self-dual) 2-forms.
The relation  $d\Omega^2=d\Omega_J^+$ is then an immediate consequence of:
\begin{lemma} \label{r4d}
Suppose $(M^4, g)$ is a compact Riemannian 4-manifold. Then
\begin{equation} \label{4dr=} d\Omega_g^+ =  d\Omega_g^- =d\Omega^2 \;.
\end{equation}
\end{lemma}

\vspace{0.1cm}
\noindent {\it Proof of Lemma \ref{r4d}.} By Hodge decomposition,
any 2-form $\alpha$ is written as $\alpha=\alpha_h+d\alpha_1+*d\alpha_2$, where the terms are respectively the harmonic,
the exact and the co-exact parts of $\alpha$. A form $\alpha$ (resp. $\beta$)
in $\Omega_g^+$ (resp. $\Omega_g^-$) further satisfies $d(\alpha_2-\alpha_1)=0$ (resp. $d(\beta_2+\beta_1)=0$).

Now, if we are given $\alpha\in \Omega_g^+$, then $d\alpha=d*d\alpha_2$. Take $\beta=-d\alpha_1+*d\alpha_2$.
By the observations above it is clear that $\beta\in  \Omega_g^-$ and $d\beta=d\alpha$.
Thus, we have shown $d\Omega_g^+\subset d\Omega_g^-$. The inclusion $d\Omega_g^-\subset d\Omega_g^+$ is similar to prove.
The last equality in (\ref{4dr=}) follows now from $\Omega^2 = \Omega_g^+ \oplus \Omega_g^-$. $\Box$

\vspace{0.2cm}
To finish the proof of the proposition it remains to verify the claim about the dimension of the quotient $(d\Omega^2)/(d\Omega_J^-)$.
Let $\mathcal{H}^+_g$ be the space of harmonic self-dual
forms and let
$$\mathcal{T}_g = \{ f \in C^{\infty}(M)| \exists \alpha \in \mathcal{H}^+_g , \; <\omega, \alpha> = f \}. $$
Since $\mathcal{H}^+_g$ is finite dimensional, $\mathcal{T}_g$ is a finite dimensional subspace of $C^{\infty}(M)$.
With the notations above, ${\rm dim} (\mathcal{T}_g) = b^+ - h^-_J$.
It is clear that
$$\mathcal{H}^+_g + \Omega_J^- = \mathcal{T}_g \omega \oplus \Omega_J^- \; .$$ This immediately implies
$$d\Omega^-_J = d(\mathcal{H}^+_g + \Omega_J^-) = d( \mathcal{T}_g \omega \oplus \Omega_J^-) .$$
Moreover, it can be easily seen more: for any $\beta \in \Omega^+_g$, $d\beta \in d\Omega^-_J$
if and only if $\beta \in \mathcal{H}^+_g + \Omega_J^- = \mathcal{T}_g \omega \oplus \Omega_J^-$.
This implies that the map
$$ C^{\infty}(M) \rightarrow d\Omega^+_g , \; \; f \mapsto d(f \omega) $$
descends to an isomorphism between the quotient spaces $C^{\infty}(M)/ \mathcal{T}_g$ and $(d\Omega^+_g)/(d\Omega_J^-)$.
It follows that the inclusion $d\Omega^-_J \subset d\Omega^+_g$ is strict and that the quotient is infinite dimensional.
\end{proof}

\begin{remark} {\rm Note that Proposition \ref{lemma4d} can be rephrased in terms of currents. 
A consequence is that each homology class in $H_+^J$ has infinitely many
$J$-invariant closed representatives, while each class in $H_-^J$ has a unique $J$-anti-invariant representative. 
Here $H^J_{\pm}\subset H_2(M; \mathbb R)$ are the $J$-(anti)-invariant homology groups defined by currents (see \cite{LZ2007, DLZ}).

This should be compared with the fact that each cohomology class in $H^-_J$, in dimension 4, has a unique
(necessarily harmonic) $J$-anti-invariant representative. }
\end{remark}

\begin{remark} \label{hdim} {\rm Proposition \ref{lemma4d} is no longer true in dimension higher than 4.
Indeed, if $J$ is any complex structure
on a compact manifold of dimension 6 or higher, $d\Omega^-_J$ cannot be a subset of $d\Omega^+_J$.
This is because a form $\beta \in \Omega^-_J$ is written as $\beta = \alpha + \overline{\alpha}$, with $\alpha$ a complex form
of type $(2,0)$, hence, generically, $d\beta$ contains terms of type $(3,0)$ and $(0,3)$.
But $\Omega^+_J = [\Omega^{1,1}_J]_{\mathbb{R}}$, so for an integrable $J$,
$d\Omega^+_J \subset [\Omega^{2,1}_J \oplus \Omega^{1,2}_J]_{ \mathbb{R}}$. It would be interesting to know if there exists
any almost complex manifold $(M^{2n}, J)$ with $2n \geq 6$ so that Proposition \ref{lemma4d} still holds,
or in what form the result might be generalized to higher dimensions.}
\end{remark}
As an application of Proposition \ref{lemma4d}, we compute the cohomology of some $J$-modified
de Rham-type of complexes.
Let us denote by $d^+_J$ (resp. $d^-_J$)
the composition of the differential $\Omega^1 \buildrel {d}\over\longrightarrow \Omega^2$
with the projection $\Omega^2 \rightarrow \Omega^+_J$ (resp. $\Omega^2 \rightarrow \Omega^-_J$).
Note that both ${\rm Ker}(d^-_J)$ and ${\rm Ker}(d^+_J)$ contain the space of closed 1-forms
$\mathcal{Z}^1$. Replacing the $\Omega^1$ and $\Omega^2$ terms in the de Rham differential complex,
we consider the following $J$-modified complexes
\begin{equation} \label{longcx+}
0 \longrightarrow \Omega^0 \buildrel {d}\over\longrightarrow {\rm Ker}(d^-_J)
\buildrel {d}\over\longrightarrow \Omega^+_J \buildrel {d}\over\longrightarrow \Omega^3
\buildrel {d}\over\longrightarrow \Omega^4
\longrightarrow 0 \; ,
\end{equation}
\begin{equation} \label{longcx-}
0 \longrightarrow \Omega^0 \buildrel {d}\over\longrightarrow {\rm Ker}(d^+_J)
\buildrel {d}\over\longrightarrow \Omega^-_J \buildrel {d}\over\longrightarrow \Omega^3
\buildrel {d}\over\longrightarrow \Omega^4
\longrightarrow 0 \; .
\end{equation}
The following are immediate observations (using also Proposition \ref{lemma4d}).
\begin{prop} \label{Jcomplexes} Suppose $(M^4,J)$ is a compact almost complex 4-manifold.

\vspace{0.1cm}

(i) The group $H^+_J$ (resp. $H^-_J$) is the cohomology group at $\Omega^+_J$-level for the complex (\ref{longcx+})
(resp. at $\Omega^-_J$-level for the complex (\ref{longcx-})).

\vspace{0.1cm}

(ii) For the complex (\ref{longcx+}),  the cohomology groups at levels $\Omega^0$, ${\rm Ker}(d^-_J)$, $\Omega^3$, $\Omega^4$ are the usual de Rham cohomology groups
$H^i(M, \mathbb{R})$ for $i = 0, 1, 3, 4$, respectively.

\vspace{0.1cm}

(iii) For the complex (\ref{longcx-}), the cohomology groups at levels $\Omega^0$, ${\rm Ker}(d^+_J)$, $\Omega^4$ are the usual de Rham cohomology groups
$H^i(M, \mathbb{R})$ for $i = 0, 1, 4$, respectively. At the $\Omega^3$-level, the cohomology group of the complex (\ref{longcx-}) is given by
$\mathcal{Z}^3/ d\Omega^-_J$, so it is infinite dimensional, by Proposition \ref{lemma4d}.
\end{prop}

\begin{remark} {\rm All statements of Proposition \ref{Jcomplexes} are still true in higher dimensions,
except those about the cohomology at the $\Omega^3$-level which use Proposition \ref{lemma4d}.
Note also that for a compact 4-manifold
${\rm Ker}(d^+_J) = \mathcal{Z}^1$, so the second differential in the complex (\ref{longcx-}) is just the zero map.}
\end{remark}

\section{Reformulations of Donaldson's question}
Given a compact almost complex manifold $(M^{2n}, J)$, let $\Omega_J^{\oplus}$ denote the cone of $J$-invariant,
$J$-positive forms, i.e. the forms $\omega \in \Omega^2$ such that $\omega(\cdot, J\cdot)$ defines a Riemannian metric
on $M$. Let us further denote by $\mathcal{S}_J^c$, $\mathcal{S}_J^t$, the set of symplectic forms that
are $J$-compatible, respectively, $J$-tamed. It is easy to see that
$$ \mathcal{S}_J^c = \Omega_J^{\oplus} \cap \mathcal{Z}^2, \; \; \mathcal{S}_J^t = (\Omega_J^{\oplus} \oplus \Omega_J^-) \cap \mathcal{Z}^2 \; .$$
With these notations, Question \ref{Donaldson} is:
if $\mathcal{S}_J^t \neq \emptyset$ is $\mathcal{S}_J^c \neq \emptyset$ as well?

\vspace{0.1cm}

Now we get characterizations in terms of spaces of exact 3-forms. In view of Proposition \ref{lemma4d}, it is natural to ask
how the space $d\Omega_J^{\oplus}$ fits in with respect to the spaces
$d\Omega_J^+$, $d\Omega_J^-$.
\begin{question} \label{do+=d+}
Is it true that either $d\Omega_J^-\cap  d\Omega_J^{\oplus}=\emptyset \; \mbox{ or } \;   d\Omega_J^{\oplus} = d\Omega_J^+\; $?
\end{question}
\noindent Seemingly unrelated, we could also ask:

\begin{question}\label{+to-}
 If $\alpha \in \Omega_J^-$ satisfies $d\alpha\in d\Omega_J^{\oplus}$,
 is it true that $d(-\alpha)\in d\Omega_J^{\oplus}$ as well?
\end{question}
\noindent It turns out that both Questions \ref{do+=d+} and \ref{+to-} are equivalent to
 Donaldson's ``tamed to compatible" question.

\begin{prop}\label{Q=Q} Suppose $(M^{2n}, J)$ is a compact almost complex manifold.
Then Questions \ref{Donaldson}, \ref{do+=d+} and \ref{+to-} are all equivalent for the given $J$.

\end{prop}
\begin{proof} For the equivalence of Questions \ref{Donaldson} and \ref{do+=d+}, first note
that $J$ is tamed by a symplectic form if and only if $d\Omega_J^-\cap  d\Omega_J^{\oplus} \neq \emptyset$.
Indeed, suppose that $J$ is tamed by the symplectic form $\omega$, which we decompose into its $J$-invariant and
$J$-anti-invariant parts $\omega =\omega_J^+ + \omega_J^-$. By the tameness assumption, $\omega_J^+ \in \Omega_J^{\oplus}$, and since
$d\omega = 0$, we have $d(\omega_J^+) = d(-\omega_J^-) \in d\Omega_J^-\cap  d\Omega_J^{\oplus}$.
Thus, $d\Omega_J^-\cap  d\Omega_J^{\oplus} \neq \emptyset$.
Showing that $d\Omega_J^-\cap  d\Omega_J^{\oplus} \neq \emptyset$ implies that $J$ is tamed is
done by just reversing this argument.

Next, observe that $J$ is compatible with some symplectic form if and only if $  d\Omega_J^{\oplus} = d\Omega_J^+ $.
Indeed, if $  d\Omega_J^{\oplus} = d\Omega_J^+ $, since $0 \in d\Omega_J^+$, it follows that
there is $\omega \in \Omega_J^{\oplus}$ so that $d\omega = 0$. For the other implication, note that
$\Omega_J^{\oplus}- \Omega_J^{\oplus}= \Omega_J^+$, that is any $J$-invariant form can be written
as the difference of two $J$-positive forms. Indeed, if $\theta \in \Omega_J^+$, taking any
$\omega \in \Omega_J^{\oplus}$, then $(n \omega + \theta) - n \omega = \theta$ and if $n$ is large enough
$n\omega + \theta \in \Omega_J^{\oplus}$. If, additionally, we can choose $\omega$ with $d\omega = 0$, then
$d(n \omega + \theta) = d \theta$, so the equality $  d\Omega_J^{\oplus} = d\Omega_J^+ $ is proved.

 \vspace{0.1cm}

Now we show that Question \ref{+to-} implies Question \ref{Donaldson}.
Suppose that $J$ is tamed by the symplectic form $\omega=\omega_J^++\omega_J^-$. Then $\omega_J^+ \in \Omega_J^{\oplus}$
and $d(-\omega_J^-) = d\omega_J^+ \in d\Omega_J^{\oplus}$.
Question \ref{+to-} implies that there exists a $J$-positive form $\tilde{\omega}_J^+$ such that $d\tilde{\omega}_J^+=d\omega_J^-$.
This implies $d(\omega_J^++\tilde{\omega}_J^+)=0$, so $\omega_J^++\tilde{\omega}_J^+$ is a symplectic form compatible with $J$.

Conversely, suppose that Question \ref{Donaldson} is true.  If for an $\alpha\in \Omega_J^-$ there exists $\beta \in \Omega_J^{\oplus}$
such that $d\alpha=d\beta$, then $\beta-\alpha$ is a symplectic form taming $J$. From Question \ref{Donaldson}, there exists
a symplectic form $\omega$ compatible with $J$. Further for a large number $n$, $n\omega-\beta$ is $J$-positive,
and  $d(n\omega-\beta)=d(-\alpha)$, which implies Question \ref{+to-}.
\end{proof}

Combining Propositions \ref{Q=Q} and \ref{lemma4d}, we get:
\begin{prop} \label{akJ4d} For a compact almost complex 4-manifold $(M^4,J)$, $J$ admits a compatible symplectic structure
if and only if $d\Omega_J^{\oplus}=d\Omega^2$.
\end{prop}

The proof of Theorem \ref{main} is now immediate.

\vspace{0.2cm}

\noindent {\it Proof of Theorem \ref{main}.} As in the proof of Proposition \ref{Q=Q}, observe that given $\alpha \in \Omega_J^-$,
the existence of a $J$-tamed symplectic form whose $J$-anti-invariant part is $\alpha$ is equivalent
to $d(-\alpha) \in d\Omega_J^{\oplus}$. The statement then follows from Proposition \ref{akJ4d}. $\Box$

\vspace{0.2cm}

\begin{remark} {\rm (i) With the notations above, Theorem \ref{main} can be restated as:
$\mathcal{S}_J^c \neq \emptyset$ if and only if the projection map $\Omega_J^{\oplus} \oplus \Omega_J^- \rightarrow \Omega_J^-$
when restricted to $\mathcal{S}_J^t$ is still onto.

\vspace{0.1cm}

(ii) In \cite{LZ2007}, Tian-Jun Li and the second author found the precise relationship between
the tame and compatible {\em cohomology} cones. Denoted by $\mathcal{K}^t_J$, $\mathcal{K}^c_J$,
these are, respectively, the images of $\mathcal{S}_J^t$ and $\mathcal{S}_J^c$ under the natural map $\mathcal{Z}^2 \rightarrow H^2(M, \mathbb{R})$.
In dimension 4, the main result of \cite{LZ2007} is:
$$\mbox{ if $\mathcal{S}_J^c \neq \emptyset$, then } \mathcal{K}^t_J = \mathcal{K}^c_J + H^-_J \; .$$
This is valid in higher dimensions as well under the additional assumption $H^+_J + H^-_J = H^2(M, \mathbb{R})$
(which holds automatically in dimension 4, see \cite{DLZ}).}
\end{remark}
As in Remark \ref{hdim}, note that Theorem \ref{main} is no longer true in dimensions higher than 4. However, this
restatement of part of Proposition \ref{Q=Q} clearly holds.
\begin{theorem} \label{mainhd} Let $(M^{2n}, J)$ be a compact almost complex manifold and assume that $J$ admits
tamed symplectic forms. The following are equivalent:

(i) $J$ admits a compatible symplectic form;

(ii) For any $\alpha \in \Omega_J^-$, if there exists a $J$-tamed symplectic form whose
$J$-anti-invariant part is $\alpha$, then there also exists a $J$-tamed symplectic form whose
$J$-anti-invariant part is $-\alpha$.

\end{theorem}


\begin{thebibliography}{99}
\bibitem{AT} D. Angella, A. Tomassini, \textit{On Cohomological Decomposition of Almost-Complex Manifolds and Deformations},
J. Symplectic Geom. 9 (2011), 403-428.

\bibitem{D} S. K. Donaldson, \textit{Two-forms on
four-manifolds and elliptic equations}, Inspired by S. S. Chern,
153--172, Nankai Tracts Math., 11, World Sci. Publ., Hackensack, NJ, 2006.
\bibitem{DLZ} T. Draghici, T.-J. Li, W. Zhang, \textit{Symplectic forms and cohomology decomposition of almost complex 4-manifolds},  Int. Math. Res. Not. IMRN 2010, no. 1, 1--17.
\bibitem{DLZ2} T. Draghici, T.-J. Li, W. Zhang, \textit{On the $J$-anti-invariant cohomology of almost complex 4-manifolds}, arXiv:1104.2511, to appear in Quart. J. Math.
\bibitem{EFV} N. Enrietti, A. Fino, L. Vezzoni, \textit{Tamed symplectic forms and SKT metrics}, arXiv 1009.0620, to appear in  J. Symplectic Geom.
\bibitem{FT} A. Fino, A. Tomassini, \textit{On some cohomological
properties of almost complex manifolds}, J. Geom. Anal. (2010) 20: 107--131.
\bibitem{Lej} M. Lejmi, \textit{Strictly nearly K\"ahler 6-manifolds are not compatible with symplectic forms},
 C. R. Math. Acad. Sci. Paris 343 (2006), no. 11-12, 759--762.
\bibitem{LZ2007} T.-J. Li, W. Zhang, \textit{Comparing tamed and compatible
symplectic cones and cohomological properties of almost complex manifolds}, Comm. Anal. Geom. 17 (2009), no. 4, 651--683.
\bibitem{LZ2010} T.-J. Li, W. Zhang, \textit{J-symplectic cones of rational four manifolds}, preprint.
\bibitem{ST} J. Streets, G. Tian, \textit{A Parabolic flow of pluriclosed metrics}, Int. Math. Res. Notices 2010 (2010), 3101--3133.
\bibitem{Pet} T. Peternell,  \textit{Algebraicity criteria for compact complex manifolds}, Math. Ann. 275 (1986), no. 4, 653--672.
\bibitem{Tau09} C. Taubes, \textit{Tamed to compatible: Symplectic forms via
moduli space integration}, J. Symplectic Geom. 9 (2011), 161--250.


\end{thebibliography}
\end{document}